\documentclass[11pt]{amsart}
\usepackage[T1]{fontenc}
\usepackage{lmodern}
\usepackage{tikz}
\usepackage{float}
\usetikzlibrary{patterns,arrows.meta}
\usepackage{amsmath,amssymb,amsthm,mathtools}
\usepackage[margin=1.03in]{geometry}
\usepackage{microtype}
\usepackage{xcolor}
\usepackage[colorlinks=true,linkcolor=blue,citecolor=blue,urlcolor=blue]{hyperref}
\mathtoolsset{showonlyrefs}

\newtheorem{theorem}{Theorem}[section]
\newtheorem{proposition}[theorem]{Proposition}
\newtheorem{lemma}[theorem]{Lemma}

\theoremstyle{definition}
\newtheorem{definition}[theorem]{Definition}
\theoremstyle{remark}
\newtheorem{remark}[theorem]{Remark}
\theoremstyle{plain}

\newcommand{\R}{\mathbb R}
\newcommand{\Sph}{\mathbb S}

\newcommand{\supp}{\operatorname{supp}}
\newcommand{\conv}{\operatorname{conv}}
\newcommand{\abs}[1]{\lvert #1\rvert}
\newcommand{\norm}[1]{\lVert #1\rVert}
\newcommand{\ip}[2]{\langle #1,#2\rangle}
\newcommand{\1}{\mathbf 1}
\newcommand{\gauss}{\gamma_n}
\newcommand{\calC}{\mathcal C}

\title[Isomorphic Busemann--Petty for arbitrary measures]
{Isomorphic Busemann--Petty for arbitrary measures: the sharp order}
\author{Alexander Koldobsky}
\author{Artem Zvavitch}
\thanks{A.K. is supported in part by the U.S. National Science Foundation
Grant DMS-2450745; A.Z. is supported in part by U.S. National Science Foundation
Grant DMS-2247771 and by the Miller Foundation at the University of Missouri.}

\subjclass[2020]{Primary: 52A20, 52A40;
Secondary: 52A23, 60D05}
\keywords{Busemann--Petty problem, convex bodies, arbitrary measures,
hyperplane sections, slicing inequalities, Gaussian mixtures, random
polytopes, random rounding}
\date{{July 31, 2026}}

\begin{document}
\begin{abstract}\begingroup
Let $\calC_n$ be the optimal constant with the following property. For every
even, continuous, strictly positive density $f$ on $\R^n$ and all
origin-symmetric convex bodies $K,L\subset\R^n$, the inequalities
\[
   \int_{K\cap\xi^\perp}f
   \leq
   \int_{L\cap\xi^\perp}f
   \qquad\text{for all }\xi\in\Sph^{n-1}
\]
imply $\int_Kf\leq\calC_n\int_Lf$. It was proved in \cite{KZ15} that
$\calC_n\leq\sqrt n$. In this paper, we prove the matching lower bound
$\calC_n\geq c\sqrt n$. To simplify the exposition, we first give a complete
one-scale construction, based on earlier work of Klartag and Koldobsky, which yields
$\calC_n\geq c\sqrt{n/\log n}$. For the sharp result, we use the
random-rounding construction of Klartag and Livshyts as a black box and
combine it with a spherical-averaging support-separation argument.
\endgroup
\end{abstract}

\maketitle

\section{Introduction and main results}
\label{sec:intro}

Throughout, a convex body is a compact convex set with nonempty interior.
For $x,y\in\R^m$, $\abs{x}$ denotes the Euclidean norm of $x$, and
$\langle x,y\rangle$ denotes the Euclidean inner product. For a measurable
set $A$, $\abs{A}$ denotes its Lebesgue measure in the dimension determined
by the context. We write $B_2^m$ for the Euclidean unit ball in $\R^m$ and
$\Sph^{m-1}=\partial B_2^m$ for the unit sphere. The vectors
$e_1,\ldots,e_m$ form the standard orthonormal basis of $\R^m$.
The symbol $\sigma_k$ denotes the $k$-dimensional surface measure on
$\Sph^k$, while $\supp f$ denotes the closure of
$\{x:f(x)\neq0\}$. All unnamed constants $c,c_0,c_1,\ldots$ and
$C,C_0,C_1,\ldots$ are positive and absolute; their values may change
from line to line.

If $f$ is a continuous nonnegative function on $\R^n$, write
\[
   \mu_f(A)=\int_A f(x)\,dx
\]
for the measure in $\R^n$ with density $f.$  For a hyperplane $H$, the notation
$\mu_f(A\cap H)$ means
\[
   \mu_f(A\cap H)=\int_{A\cap H} f(x)\,d\lambda_H(x),
\]
where $\lambda_H$ is $(n-1)$-dimensional Lebesgue measure on $H$.  Thus the
same ambient density is restricted to the hyperplane; no conditional
normalization is made.

The starting point is the classical Busemann--Petty problem, posed by
Busemann and Petty \cite[Problem~1]{BP56}.  It asks whether, for
origin-symmetric convex bodies $K,L\subset\R^n$, the inequalities
\[
   \abs{K\cap\xi^\perp}
   \leq
   \abs{L\cap\xi^\perp}
   \qquad\text{for all }\xi\in\Sph^{n-1}
\]
necessarily imply $\abs K\leq\abs L$.  Here
$
   \xi^\perp=\{x\in\R^n:\ip{x}{\xi}=0\}
$ 
is the hyperplane through the origin orthogonal to $\xi$.  The answer is
affirmative for
$n\leq4$ and negative for $n\geq5$; see
\cite{GKS99,Zhang99,Gardner06,Koldobsky05} and the references therein.

The second-named author extended the problem from Lebesgue measure to
essentially arbitrary measures \cite{Zvavitch05}.  In the symmetric setting
relevant here, one asks whether
$
   \mu_f(K\cap\xi^\perp)
   \leq
   \mu_f(L\cap\xi^\perp)$
  for all $ \xi\in\Sph^{n-1}
$
implies $\mu_f(K)\leq\mu_f(L)$, where the same density is used for the
full-dimensional measures and for the hyperplane sections.  For even,
continuous, strictly positive densities, the dimensional dichotomy is again affirmative for $n\leq4$ and negative for $n\geq5$; see
\cite[Corollary~2]{Zvavitch05}.

Since the exact implication fails in dimensions $n\geq5$, the authors
introduced its isomorphic version in
\cite[Section~1]{KZ15} as follows.

\begin{definition}\label{def:Cn}
For $n\geq2$, let $\calC_n$ be the infimum of all $C>0$ such that the
following assertion holds.  Whenever $f:\R^n\to(0,\infty)$ is even and
continuous and
$K,L\subset\R^n$ are origin-symmetric convex bodies satisfying
\begin{equation}\label{eq:section-comparison}
   \mu_f(K\cap\xi^\perp)
   \leq
   \mu_f(L\cap\xi^\perp)
   \qquad\text{for all }\xi\in\Sph^{n-1},
\end{equation}
then $\mu_f(K)\leq C\mu_f(L).$
\end{definition}
In an earlier paper \cite[Theorem~1, pp.~264--266]{KZ15}, the authors proved that
\begin{equation}\label{eq:known-upper}
       \calC_n\leq\sqrt n.
\end{equation}
The result there is stated for
even continuous densities.  Requiring strict positivity does not affect the
upper bound.

Related isomorphic comparison theorems in which the two section
integrals are formed with different measures, together with applications to
slicing and distance inequalities, were obtained by Koldobsky, Paouris, and
Zvavitch \cite{KPZ22}. Extensions to the case where the measures for sections
and for the whole body are different were suggested in \cite{Zvavitch05, GHS26}.

For Lebesgue measure, the isomorphic Busemann--Petty problem is directly
connected with Bourgain's slicing problem.  The formulation through central
hyperplane sections appears in
\cite[Section~2, especially Lemma~2 and the following remark]{Bourgain86},
while the comparison relevant here is due to Milman and Pajor
\cite[Section~5]{MP89}.  In particular, the isomorphic Busemann--Petty
constant for volume is controlled, up to an absolute factor, by the maximal
isotropic constant in the corresponding dimension.  The assertion that the
isotropic constants of convex bodies are uniformly bounded in all dimensions
is an equivalent formulation of Bourgain's slicing problem; see
\cite[Section~5]{MP89} and \cite[Section~1]{KL25}.  Following Chen's
breakthrough \cite[Theorem~1]{Chen21}, the polylogarithmic estimate of
Klartag and Lehec \cite[Theorem~1.1]{KL22}, and Guan's bound
\cite[Theorem~1.1]{Guan24}, Klartag and Lehec established this
dimension-free bound in \cite[Theorem~1.2]{KL25};  see also
\cite{Be26} for an alternative proof. Consequently, the
isomorphic Busemann--Petty constant for volume is bounded by an absolute
constant.

The related one-body slicing inequality for arbitrary measures
has a different dimensional behavior.  Koldobsky proved its general
$O(\sqrt n)$ upper bound \cite[Theorem~1]{Koldobsky14}; see also
\cite{Koldobsky15} for estimates in terms of the outer volume ratio and a different proof
of the $\sqrt{n}$ estimate in \cite{CGL}.  
Klartag and Koldobsky constructed examples yielding the lower bound
$
   c\sqrt{n/\log\log n}
$ 
\cite[Theorems~1.1 and~1.3]{KK18}.  Klartag and Livshyts subsequently
obtained the sharp lower bound $c\sqrt n$ by means of their random-rounding
construction \cite[Theorem~1.1]{KL20}.  Thus, unlike the corresponding
problem for volume, the optimal constant in the slicing inequality for
arbitrary measures is of order $\sqrt n$.

In closely related work, Bobkov, Klartag, and Koldobsky proved moment and
slicing estimates for arbitrary measures and combined their slicing
inequality with the example from \cite{KK18} to obtain lower bounds for the
outer-volume-ratio distance from an arbitrary symmetric convex body in $\R^n$ to the classes 
of unit balls of $n$-dimensional subspaces of $L_p$; see \cite{BKK18}.

The same dimensional contrast occurs in the two-body problem studied here:
the lower bound below grows with the dimension.

\begin{theorem}\label{thm:main}
There is an absolute constant $c>0$ such that, for every $n\geq2$,
\[
   c\sqrt n\leq\calC_n\leq\sqrt n.
\]
More explicitly, for all sufficiently large $n$ there are origin-symmetric
convex bodies $K_n,L_n\subset\R^n$ and an even, continuous, strictly positive,
integrable density $f_n$ such that \eqref{eq:section-comparison} holds and
\[
   \frac{\mu_{f_n}(K_n)}{\mu_{f_n}(L_n)}
   \geq c\sqrt n.
\]
\end{theorem}

For clarity of exposition, we first present a simpler one-scale argument.  It
gives a lower bound within a factor $\sqrt{\log n}$ of the optimal one and
introduces the main geometric mechanism without the multiscale notation.

\begin{theorem}[The one-scale lower bound]\label{thm:one-scale}
There is an absolute constant $c>0$ such that
\[
       \calC_n\geq c\sqrt{\frac{n}{\log n}}
\]
for every $n\geq3$.
\end{theorem}

Here a \emph{scale} means a radial level in the random construction.  The
proof of Theorem~\ref{thm:one-scale} is a one-scale specialization of the
construction in \cite[Section~3]{KK18}: all Gaussian centers come from one
family of random directions placed at a single radius.  More precisely, the
random-direction selection, the random polytope, and the Gaussian-mixture
estimate are adapted from \cite[Section~3]{KK18}.  We include complete proofs
 of the spherical-coordinate estimate, the concentration and
random-selection arguments, and the Gaussian marginal calculation.  Gluskin's
polytope-volume estimate is the only non-elementary volumetric input.

The geometric part added here is a support-separation and gluing procedure.
One first constructs a probability density $p$ supported in an outer convex
body $T$ but outside a Euclidean ball $L$, with every central hyperplane
integral at most $C/\sqrt n$.  A second probability density $h$, supported in
$L\setminus T$, has uniformly large central hyperplane integrals.  A suitably
small multiple of $h$ makes every central section of $L$ larger than the
corresponding section of $T$, while keeping the total measure of $L$ small (see Figure \ref{fig:support-separation}).
\begin{figure}[H]
\centering
\begin{tikzpicture}[scale=1.05, line join=round]
  \coordinate (O) at (0,0);

  \begin{scope}
    \clip (O) circle (2.25);
    \path[fill=blue!16, even odd rule]
      (-3.7,-2.8) rectangle (3.7,2.8)
      (-3.25,0) --
      (-2.15,1.05) --
      (-0.75,1.65) --
      (0,1.82) --
      (0.75,1.65) --
      (2.15,1.05) --
      (3.25,0) --
      (2.15,-1.05) --
      (0.75,-1.65) --
      (0,-1.82) --
      (-0.75,-1.65) --
      (-2.15,-1.05) -- cycle;
  \end{scope}

  \begin{scope}
    \clip
      (-3.25,0) --
      (-2.15,1.05) --
      (-0.75,1.65) --
      (0,1.82) --
      (0.75,1.65) --
      (2.15,1.05) --
      (3.25,0) --
      (2.15,-1.05) --
      (0.75,-1.65) --
      (0,-1.82) --
      (-0.75,-1.65) --
      (-2.15,-1.05) -- cycle;
    \path[fill=red!11, even odd rule]
      (-3.7,-2.8) rectangle (3.7,2.8)
      (O) circle (2.25);
  \end{scope}

  \draw[blue!65!black, very thick]
      (O) circle (2.25);

  \draw[black!70, very thick]
      (-3.25,0) --
      (-2.15,1.05) --
      (-0.75,1.65) --
      (0,1.82) --
      (0.75,1.65) --
      (2.15,1.05) --
      (3.25,0) --
      (2.15,-1.05) --
      (0.75,-1.65) --
      (0,-1.82) --
      (-0.75,-1.65) --
      (-2.15,-1.05) -- cycle;

  \foreach \x/\y in {
      2.70/0.16, 2.52/-0.25, 2.42/0.42,
      -2.70/-0.16, -2.52/0.25, -2.42/-0.42}
    \shade[
      inner color=red!80,
      outer color=red!8
    ] (\x,\y) circle (0.13);

  \foreach \x/\y/\a in {
      0/2.04/0, 0.78/1.91/-15, -0.78/1.91/15,
      0/-2.04/0, 0.78/-1.91/15, -0.78/-1.91/-15}
    \fill[
      blue!65,
      opacity=0.48,
      rotate around={\a:(\x,\y)}
    ] (\x,\y) ellipse (0.30 and 0.10);

  \fill (O) circle (0.025);
  \node[below right] at (O) {$0$};

  \node[blue!65!black, anchor=west]
      at (1.55,2.28) {$L\cap\xi^\perp$};

  \node[black!75, anchor=west]
      at (2.52,1.05) {$T\cap\xi^\perp$};

  \node[red!70!black, anchor=west]
      at (2.82,-0.55) {$\operatorname{supp}p$};
  \draw[->, red!70!black]
      (2.82,-0.48) -- (2.55,-0.25);

  \node[blue!70!black, anchor=east]
      at (-1.22,2.46) {$\operatorname{supp}h$};
  \draw[->, blue!70!black]
      (-1.17,2.38) -- (-0.78,1.97);

  \node[black!55, anchor=west]
      at (-3.45,-2.55)
      {\small schematic view in $\xi^\perp$};
\end{tikzpicture}

\caption{Schematic representation of the support-separation and gluing
procedure inside a central hyperplane $\xi^\perp$. The density $p$ is
supported in $T\setminus L$ and has small central hyperplane integrals,
whereas $h$ is supported in $L\setminus T$ and has uniformly large central
hyperplane integrals.}
\label{fig:support-separation}
\end{figure}
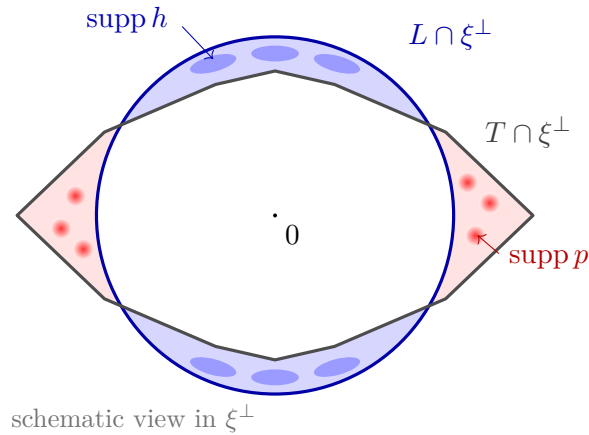

The one-scale construction gives
$
       \abs T^{1/n}\leq C\sqrt{\log n}.$ 
Accordingly, the comparison ball has radius $A\sqrt{n\log n}$, and the central
hyperplane integrals of $h$ are of order $1/\sqrt{\log n}$.

For the sharp application, we use the construction of Klartag and Livshyts
\cite{KL20} only through a black-box consequence of their proof.  It supplies
an outer body $T$ with bounded volume radius and an even probability density
that gives positive mass to $T$ and has central hyperplane integrals of order
$n^{-1/2}$.  A spherical-averaging identity then shows that only
$O(n^{-1/2})$ of this mass can lie in a fixed ball of radius $A\sqrt n$.
Continuous cutoffs therefore produce the required density $p$ supported in
$T\setminus A\sqrt nB_2^n$.  This argument uses neither the individual
Gaussian components nor the multiscale parameters in \cite{KL20}.

We would like to stress that Theorem~\ref{thm:main} is not a formal consequence of the slicing example in
\cite{KL20}: the present problem requires one density and two bodies with
ordered sections but reversed total masses.  The new support-separation
lemma, the annular comparison density, and the gluing principle convert the
Klartag--Livshyts example into such a two-body comparison.

\noindent\textbf{Organization of the paper.}
Sections~\ref{sec:one-scale}--\ref{sec:one-application} give the complete
one-scale proof.  Section~\ref{sec:KL} uses  the main result of \cite{KL20} as a
black box and derives the support separation needed for the sharp lower
bound.  Section~\ref{sec:sharp-application} completes its proof.

\section{{The one-scale random construction}}
\label{sec:one-scale}
We first record the only non-elementary volumetric input needed
in the one-scale argument.  Gluskin proved that if a centrally symmetric polytope
$Q\subset B_2^n$ has at most $2M$ vertices, then
\[
   \left(\frac{\abs Q}{\abs{B_2^n}}\right)^{1/n}
   \leq
   C\min\left\{1,\frac1{\sqrt n}
       \sqrt{1+\log(M/n)}\right\};
\]
see \cite[introduction, equation~(1)]{Glu89}.  Combining this with
$\abs{B_2^n}^{1/n}\leq C/\sqrt n$ (see, for example,
Lemma~\ref{lem:ball-estimates} below) gives the formulation below.  The
nonsymmetric formulation follows because the convex hull of $M$ points is
contained in their absolute convex hull, which has at most $2M$ vertices.
See also \cite[Chapter~4]{Pisier89} and \cite[Chapter~6]{AGM21} for the role
of Gluskin polytopes in finite-dimensional Banach-space theory.

\begin{theorem}[Gluskin's polytope-volume estimate]\label{thm:gluskin}
There is an absolute constant $C>0$ such that the following holds.
If $M\geq2n$ and $Q\subset B_2^n$ is the convex hull of at most $M$ points,
then
\[
   \abs{Q}^{1/n}
   \leq
   \frac{C}{n}\sqrt{\log\!\left(1+\frac{M}{n}\right)}.
\]
The same conclusion, with a change of the absolute constant, holds when $Q$
is the absolute convex hull of at most $M$ points of $B_2^n$.
\end{theorem}

We shall also use the following classical estimates for the volumes and surface area of
Euclidean balls.

\begin{lemma}[Volumes of Euclidean balls]\label{lem:ball-estimates}
For every integer $m\geq2$,
\begin{equation}\label{eq:ball-estimates}
   \frac{c}{\sqrt m}
   \leq
   \abs{B_2^m}^{1/m}
   \leq
   \frac{C}{\sqrt m},
   \qquad
   c\sqrt m
   \leq
   \frac{\abs{B_2^{m-1}}}{\abs{B_2^m}}
   \leq
   C\sqrt m, \qquad c\sqrt m \leq \frac{\sigma_{m-2}(\Sph^{m-2})}
   {\sigma_{n-1}(\Sph^{m-1})}
   \leq C\sqrt m.
\end{equation}
\end{lemma}
Those estimates follow directly from
$\abs{B_2^m}=\pi^{m/2}/\Gamma(m/2+1)$,  $\sigma_{m-1}(\Sph^{m-1})=m\abs{B_2^m}$ and the standard Stirling and
gamma-ratio estimates; see \cite[Section~5.6(i),
formula~(5.6.4)]{NIST} and \cite[Appendix~A]{Pisier89}.

We shall use the standard Gaussian probability density
\begin{equation}\label{eq:standard-Gaussian}
   \gauss(x)
   =(2\pi)^{-n/2}\exp\!\left(-\frac{\abs{x}^2}{2}\right),
   \qquad x\in\R^n.
\end{equation}
For a unit vector $\xi$, a point $z\in\R^n$, and $t\in\R$, direct integration
over the affine hyperplane $\xi^\perp+t\xi$ gives
\begin{equation}\label{eq:Gaussian-affine-section}
 \int_{\xi^\perp+t\xi}\gauss(x-z)\,
       d\lambda_{\xi^\perp+t\xi}(x)
 =
 \frac1{\sqrt{2\pi}}
 \exp\!\left[-\frac{(t-\ip{z}{\xi})^2}{2}\right].
\end{equation}

We shall use the following compact-support truncation, so we
isolate it here.
Fix a continuous function $\chi:[0,\infty)\to[0,1]$ such that
\[
\chi(s)=
\begin{cases}
1, & 0\leq s\leq\sqrt{2},\\[2mm]
\dfrac{2-s}{2-\sqrt{2}}, & \sqrt{2}<s<2,\\[2mm]
0, & s\geq2,
\end{cases}
\]
and set
\begin{equation}\label{eq:common-mchi}
   m_\chi=\int_{\R^n}\gauss(y)
       \chi\!\left(\frac{\abs y}{\sqrt n}\right)\,dy.
\end{equation}

\begin{lemma}[Truncation of an even Gaussian mixture]
\label{lem:Gaussian-truncation}
The normalizing constant $m_\chi$ defined in
\eqref{eq:common-mchi} satisfies
\begin{equation}\label{eq:common-mchi-bounds}
       \frac12\leq m_\chi\leq1.
\end{equation}
Let $z_1,\ldots,z_M\in\R^n$ and define
\[
   g_Z(x)=\frac1{2M}\sum_{j=1}^M
   \bigl[\gauss(x-z_j)+\gauss(x+z_j)\bigr].
\]
Then the function
\begin{equation}\label{eq:truncated-mixture}
   p_Z(x)=\frac1{2Mm_\chi}\sum_{j=1}^M
   \bigg[
   \gauss(x-z_j)
   \chi\!\left(\frac{\abs{x-z_j}}{\sqrt n}\right)
   +
   \gauss(x+z_j)
   \chi\!\left(\frac{\abs{x+z_j}}{\sqrt n}\right)
   \bigg]
\end{equation}
is an even continuous probability density satisfying
$0\leq p_Z\leq2g_Z$ and
\begin{equation}
   \supp p_Z\subset
   \bigcup_{j=1}^M
   \bigl[
      (z_j+2\sqrt nB_2^n)
      \cup
      (-z_j+2\sqrt nB_2^n)
   \bigr].
   \label{eq:truncation-support}
\end{equation}
Consequently, for every $\xi \in \Sph^{n-1}$,
$\int_{\xi^\perp}p_Z\leq2\int_{\xi^\perp}g_Z.$
\end{lemma}
\begin{proof}
If $G$ is a standard Gaussian vector in $\R^n$, then $
   m_\chi=\mathbb E\,\chi\!\left(\frac{\abs G}{\sqrt n}\right).
$
Since $\chi=1$ on $[0,\sqrt2]$ and $\mathbb E\abs G^2=n$, Markov's
inequality gives
\[
   m_\chi\geq\mathbb P\{\abs G^2\leq2n\}
   \geq1-\frac{\mathbb E\abs G^2}{2n}=\frac12.
\]
The upper bound $1 \ge m_\chi$ follows from $0\leq\chi\leq1$.  Translation invariance shows
that every truncated Gaussian in \eqref{eq:truncated-mixture} has integral
$m_\chi$; hence $p_Z$ is a probability density.  Its evenness, continuity,
and support inclusion are immediate from the definition.  Finally,
\eqref{eq:common-mchi-bounds} and $0\leq\chi\leq1$ give the pointwise bound
$p_Z\leq2g_Z$, and integration over $\xi^\perp$ proves the last assertion.
\end{proof}
For $t\in\R$, set
\[
   \varphi(t)=e^{-t^2/2}.
\]

{
\begin{lemma}[The distribution of a spherical coordinate]
\label{lem:spherical-mean}
Let $n\geq3$, let $\Theta$ be uniformly distributed on $\Sph^{n-1}$, and fix
$\xi\in\Sph^{n-1}$.  There are absolute constants $c,C>0$ such that
\[
   \frac{c}{\sqrt n}
   \leq
   \mathbb E\varphi\!\left(n\ip{\Theta}{\xi}\right)
   \leq
   \frac{C}{\sqrt n}.
\]
\end{lemma}

\begin{proof}
The upper estimate also follows from the more general affine estimate
\cite[Lemma~3.2]{KK18}, by taking $R=n$ and $t=0$.  We give a complete
 two-sided proof because the lower estimate is needed when
Lemma~\ref{lem:bernstein} is applied below.

By rotational invariance and the standard spherical-coordinate
formula, the random variable $s=\ip{\Theta}{\xi}$ has density
\[
   \frac{\sigma_{n-2}(\Sph^{n-2})}
   {\sigma_{n-1}(\Sph^{n-1})}
   (1-s^2)^{(n-3)/2}\1_{[-1,1]}(s).
\]
Therefore, applying Lemma \ref{lem:ball-estimates},
$$
   \mathbb E\varphi\!\left(n\ip{\Theta}{\xi}\right)
   =
   \frac{\sigma_{n-2}(\Sph^{n-2})}
   {\sigma_{n-1}(\Sph^{n-1})}\int_{-1}^{1}
   e^{-n^2s^2/2}(1-s^2)^{(n-3)/2}\,ds\leq
   C_0\sqrt n\int_{\R}e^{-n^2s^2/2}\,ds
   \leq \frac{C}{\sqrt n}.
$$
For the reverse bound, restrict the integral to $\abs{s}\leq1/n$.  On this
interval, $
   e^{-n^2s^2/2}\geq e^{-1/2}
$ 
and
\[
   (1-s^2)^{(n-3)/2}
   \geq
   (1-n^{-2})^{(n-3)/2}
   \geq c_1>0
\]
for all $n\geq3$, with an absolute constant $c_1$.  Indeed, the displayed
sequence converges to $1$, and each of its finitely many initial values is
positive.
Consequently,
\[
   \mathbb E\varphi\!\left(n\ip{\Theta}{\xi}\right)
   \geq c_0\sqrt n\cdot\frac{2}{n}\cdot e^{-1/2}c_1
   \geq\frac{c}{\sqrt n}.
\]
\end{proof}
}

\begin{lemma}[A Bernstein-type bound]\label{lem:bernstein}
Let $Y_1,\dots,Y_N$ be independent identically distributed random variables
with values in $[0,1]$, and put $p=\mathbb EY_1$.  Then
\[
   \mathbb P\left\{\frac1N\sum_{i=1}^NY_i\geq2p\right\}
   \leq e^{-cNp},
\]
where $c>0$ is an absolute constant.
\end{lemma}

\begin{proof}
A closely related form, with threshold $3p$ and upper bound $e^{-Np}$,
is proved in \cite[Lemma~3.1]{KK18}.  We reproduce the short argument,
with constants adjusted to obtain the formulation above.
Since $0\leq Y_1\leq1$, we have $Y_1^m\leq Y_1$ for every integer
$m\geq1$.  Therefore
\[
   \mathbb Ee^{Y_1}
   =
   1+\sum_{m=1}^{\infty}\frac{\mathbb EY_1^m}{m!}
   \leq
   1+p\sum_{m=1}^{\infty}\frac1{m!}
   =1+p(e-1)
   \leq e^{p(e-1)}.
\]
Markov's inequality and independence now give
\[
   \mathbb P\left\{\sum_{i=1}^NY_i\geq2Np\right\}
   \leq
   e^{-2Np}\mathbb E\exp\left(\sum_{i=1}^NY_i\right)\leq
   \exp\bigl(-(3-e)Np\bigr).
\]
\end{proof}
The next lemma is the standard volumetric estimate for a net of
the Euclidean sphere; see, for example,
\cite[Chapter~1, Section~1]{Pisier89}.
\begin{lemma}[A Euclidean net of the sphere]\label{lem:sphere-net}
For every $0<\delta\leq1$, there is a finite set
$\mathcal N\subset\Sph^{n-1}$ such that every $\xi\in\Sph^{n-1}$ satisfies $
       \min_{\eta\in\mathcal N}\abs{\xi-\eta}\leq\delta
$
and
\begin{equation}\label{eq:net-cardinality}
       \#\mathcal N
       \leq\left(1+\frac2\delta\right)^n
       \leq\left(\frac3\delta\right)^n.
\end{equation}
\end{lemma}

\begin{proposition}[A uniform choice of random directions]
\label{prop:directions}
For all sufficiently large $n$, there exist
$\theta_1,\dots,\theta_N\in\Sph^{n-1}$, where $N=n^2$, such that
\begin{equation}\label{eq:random-directions}
   \frac1N\sum_{i=1}^N
   \varphi\!\left(n\ip{\xi}{\theta_i}\right)
   \leq\frac{C}{\sqrt n}
   \qquad\text{for every }\xi\in\Sph^{n-1}.
\end{equation}
\end{proposition}

\begin{proof}
A more general affine version of this random selection and net argument is
proved in \cite[Proposition~3.4]{KK18}.  In fact, its conclusion with
$R=n$, $N=n^2$, and $t=0$ implies \eqref{eq:random-directions} for all
sufficiently large $n$; the mechanism is then iterated at two scales in
\cite[Lemma~3.6]{KK18}.  We include the complete
{one-scale} proof and make
explicit the lower spherical-mean estimate needed for the concentration
exponent.

Let $\Theta_1,\dots,\Theta_N$ be independent and uniformly distributed on
$\Sph^{n-1}$.  For fixed $\xi\in\Sph^{n-1}$, apply
Lemma~\ref{lem:bernstein} to
$
   Y_i=\varphi\!\left(n\ip{\xi}{\Theta_i}\right).
$ 
Put $p_n=\mathbb EY_i$.  By Lemma~\ref{lem:spherical-mean},
\[
       \frac{c_0}{\sqrt n}\leq p_n\leq\frac{C_0}{\sqrt n}.
\]
The event in which the average exceeds $2C_0/\sqrt n$ is contained in the
event in which it exceeds $2p_n$.  Lemma~\ref{lem:bernstein}, together with
$N=n^2$ and the lower bound for $p_n$, therefore gives
\begin{equation}\label{eq:fixed-direction-probability}
   \mathbb P\left\{
   \frac1N\sum_{i=1}^N
   \varphi\!\left(n\ip{\xi}{\Theta_i}\right)
   >\frac{2C_0}{\sqrt n}\right\}
   \leq e^{-cn^{3/2}}.
\end{equation}
Since $
       \varphi'(t)=-t e^{-t^2/2},
$
we have $\sup_{t\in\R}\abs{\varphi'(t)}=e^{-1/2}<1$.  Thus $\varphi$ is
$1$-Lipschitz.  Consequently, for every choice of
$\theta_1,\dots,\theta_N$, the function
\[
   G(\xi)=\frac1N\sum_{i=1}^N
   \varphi\!\left(n\ip{\xi}{\theta_i}\right)
\]
is $n$-Lipschitz on $\Sph^{n-1}$, because
\[
   \abs{G(\xi)-G(\eta)}
   \leq
   \frac1N\sum_{i=1}^N
   n\abs{\ip{\xi-\eta}{\theta_i}}
   \leq n\abs{\xi-\eta}.
\]
Let $\delta=n^{-3/2}$ and let $\mathcal N$ be a $\delta$-net in
$\Sph^{n-1}$ supplied by Lemma~\ref{lem:sphere-net}.  Then
\[
   \#\mathcal N\leq\left(\frac3\delta\right)^n
   \leq e^{C_1n\log n}.
\]
By \eqref{eq:fixed-direction-probability} and the union bound, the probability
that the desired estimate fails at some point of $\mathcal N$ is at most
\[
   e^{C_1n\log n}e^{-cn^{3/2}}<1
\]
for all sufficiently large $n$.  We may therefore fix a realization for
which $G(\eta)\leq2C_0/\sqrt n$ for every $\eta\in\mathcal N$.  Given
$\xi\in\Sph^{n-1}$, choose $\eta\in\mathcal N$ with
$\abs{\xi-\eta}\leq\delta$.  Then
\[
   G(\xi)\leq G(\eta)+n\delta
   \leq\frac{2C_0+1}{\sqrt n},
\]
which proves the proposition.
\end{proof}

Fix once and for all a realization
$\theta_1,\ldots,\theta_N\in\Sph^{n-1}$ satisfying
Proposition~\ref{prop:directions}.  In particular,
\(\abs{\theta_i}=1\) for every \(i\).    Define the origin-symmetric polytope
\begin{equation}\label{eq:P-definition}
   P=\conv
   \{\pm\theta_1,\ldots,\pm\theta_N,\pm e_1,\ldots,\pm e_n\}.
\end{equation} 

\begin{lemma}[Geometry of the {one-scale} polytope]\label{lem:P-geometry}
The polytope $P$ satisfies
\begin{align}
   \frac1{\sqrt n}B_2^n&\subset P, \label{eq:inball-P}\\
   \abs{P}^{1/n}
   &\leq C\frac{\sqrt{\log n}}{n}. \label{eq:P-volume}
\end{align}
Consequently, the origin-symmetric convex body
\begin{equation}\label{eq:T-definition}
       T=3nP
\end{equation}
satisfies
\begin{align}
   3\sqrt n\,B_2^n&\subset T, \label{eq:T-inradius}\\
   \abs{T}^{1/n}&\leq C\sqrt{\log n}. \label{eq:T-volume}
\end{align}
\end{lemma}
\begin{proof}
The convex hull of the $2n$ coordinate points is the cross-polytope $B_1^n$.
If $x\in n^{-1/2}B_2^n$, the Cauchy--Schwarz inequality gives
$
       \norm{x}_1
       \leq\sqrt n\,\abs{x}\leq1.
$ 
Thus
$
       \frac1{\sqrt n}B_2^n\subset B_1^n\subset P,
$
which is \eqref{eq:inball-P}.

Every  vertex of $P$ belongs to $B_2^n$, and
$P$ is the convex hull of
$
       2N+2n=2n^2+2n\leq4n^2
$
such points.  Apply Theorem~\ref{thm:gluskin} with $M=4n^2$.  Since
$
       \log\!\left(1+\frac{4n^2}{n}\right)
      \leq C\log n,
$
for $n\geq2$, we obtain \eqref{eq:P-volume}.  Scaling
\eqref{eq:inball-P} and \eqref{eq:P-volume} by $3n$ proves
\eqref{eq:T-inradius} and \eqref{eq:T-volume}.  This is the
one-scale
counterpart of the bounded-volume-radius estimate for the two-scale body in
\cite[Lemma~3.7]{KK18}.
\end{proof}

The factor $3$ in \eqref{eq:T-definition} accommodates the
truncation of the Gaussian bumps used below.  Indeed, the centers
$\pm n\theta_i$ belong to $nP$, while
$2\sqrt n\,B_2^n\subset2nP$ by \eqref{eq:inball-P}.  Hence every truncated
bump is supported in
\[
    \pm n\theta_i+2\sqrt n\,B_2^n
    \subset nP+2nP=3nP.
\]

We now place identical Gaussian bumps at the \(2N\) points
\(\pm n\theta_i\), each of which has Euclidean norm \(n\).  Define
\begin{equation}\label{eq:g-density}
   g(x)=\frac1{2N}\sum_{i=1}^N
   \left[
   \gauss(x-n\theta_i)+\gauss(x+n\theta_i)
   \right].
\end{equation}
Every translated Gaussian has integral one, so $g$ is a probability density.
The pairing of the centers $n\theta_i$ and $-n\theta_i$ shows that $g$ is
even.

\begin{lemma}[Small central sections of the Gaussian mixture]
\label{lem:g-sections}
For every $\xi\in\Sph^{n-1}$,
\[
   \int_{\xi^\perp}g(x)\,d\lambda_{\xi^\perp}(x)
   \leq\frac{C}{\sqrt n}.
\]
\end{lemma}

\begin{proof}
The same Gaussian marginal identity is used in the two-scale
construction \cite[equation~(21)]{KK18}, and the resulting affine-hyperplane
estimate is \cite[Lemma~3.9]{KK18}.  Here we only need the central case $t=0$ of
\eqref{eq:Gaussian-affine-section}.  Using $z=\pm n\theta_i$ in
\eqref{eq:g-density}, followed by
\eqref{eq:random-directions}, yields
\[
   \int_{\xi^\perp}g(x)\,d\lambda_{\xi^\perp}(x)
   =
   \frac1{N\sqrt{2\pi}}
   \sum_{i=1}^N
   \varphi\!\left(n\ip{\xi}{\theta_i}\right)
   \leq\frac{C}{\sqrt n}.
\]
\end{proof}

\section{{The support-separated density in the one-scale construction}}
\label{sec:one-outer}

We now turn $g$ into a compactly supported probability density whose support
lies in $T$ but outside a Euclidean comparison ball.  This support separation
is essential for the later comparison argument.

{
\begin{proposition}[The one-scale outer density]\label{prop:one-outer-density}
For every fixed $A>0$, there are an absolute constant $C>0$ and an integer
$n_0=n_0(A)$ such that, for every $n\geq n_0$,
there exists an even continuous probability density $p$ satisfying
\begin{equation}\label{eq:p-support}
   \supp p\subset T\setminus L,
\end{equation}
where $
L=   A\sqrt{n\log n}B_2^n,
$ and
\begin{equation}\label{eq:p-sections}
   \int_{\xi^\perp}p(x)\,d\lambda_{\xi^\perp}(x)
   \leq\frac{C}{\sqrt n}
   \qquad\text{for all }\xi\in\Sph^{n-1}.
\end{equation}
\end{proposition}
}

\begin{proof}
Apply Lemma~\ref{lem:Gaussian-truncation} to the centers
$z_i=n\theta_i$, $1\leq i\leq N$.  The corresponding mixture $g_Z$ is
exactly the density $g$ in \eqref{eq:g-density}; let $p$ be the truncated
density $p_Z$ from \eqref{eq:truncated-mixture}.  Then $p$ is an even
continuous probability density, and
       $p\leq2g.$ 
Thus Lemma~\ref{lem:g-sections} gives \eqref{eq:p-sections}.

It remains to verify the support claim \eqref{eq:p-support}.  By
\eqref{eq:truncation-support},
\[
   \supp p\subset
   \bigcup_{i=1}^N
   \left[
   \bigl(n\theta_i+2\sqrt n B_2^n\bigr)
   \cup
   \bigl(-n\theta_i+2\sqrt n B_2^n\bigr)
   \right].
\]
Since $\pm n\theta_i\in nP$ and, by \eqref{eq:inball-P},
$2\sqrt nB_2^n\subset2nP$, each translated ball is contained in
\[
       nP+2nP=3nP=T.
\]
Every point in one of these balls has Euclidean norm at least
$n-2\sqrt n$.  For fixed $A$ and all sufficiently large $n$,
\[
       n-2\sqrt n>A\sqrt{n\log n}.
\]
Hence $\supp p\cap L=\varnothing$, proving \eqref{eq:p-support}.
\end{proof}

\section{The common comparison density and gluing principles}
\label{sec:common}

{
We now isolate the part of the proof that is common to the one-scale and
sharp constructions.  The parameter $v$ below records the size of the
volume radius $\abs{Q}^{1/n}$.  The one-scale construction uses
$v=\sqrt{\log n}$, whereas the construction from \cite{KL20} uses $v=1$.
}

\begin{lemma}[Central section versus total volume]
\label{lem:central-section-volume}
Let $Q\subset\R^n$ be an origin-symmetric convex body and suppose that
$rB_2^n\subset Q$.  Then, for every $\xi\in\Sph^{n-1}$,
\[
   \abs{Q\cap\xi^\perp}
   \leq\frac{n}{2r}\abs Q.
\]
\end{lemma}

\begin{proof}
Since $\pm r\xi\in Q$ and $Q$ is convex, the two cones
$\conv(Q\cap\xi^\perp,r\xi)$ and 
   $\conv(Q\cap\xi^\perp,-r\xi)$
are contained in $Q$, and their interiors are disjoint.  Each cone has
$(n-1)$-dimensional base $Q\cap\xi^\perp$ and height $r$, and therefore has
$n$-dimensional volume $
       \frac{r}{n}\abs{Q\cap\xi^\perp}.$ 
Adding the volumes of the two cones gives
$
       \abs Q\geq\frac{2r}{n}\abs{Q\cap\xi^\perp},
$
which is equivalent to the required estimate.
\end{proof}

\begin{proposition}[Annular comparison density]
\label{prop:annular-density}
For every $r_0,C_0>0$, there exists
$A_*=A_*(r_0,C_0)>0$ such that, for every $A\geq A_*$,
there are constants $c=c(r_0,C_0,A)>0$
    and 
    $n_0=n_0(r_0,C_0,A)\in\mathbb{N}$ 
with the following property. Let $n\geq n_0$, let $1\leq v\leq\sqrt{n}$, and let
$Q\subset\R^n$ be an origin-symmetric convex body satisfying
\begin{equation}\label{eq:abstract-Q-geometry}
    r_0\sqrt{n}\,B_2^n\subset Q,
    \qquad
    \abs{Q}^{1/n}\leq C_0v.
\end{equation}
Set $ U=Av\sqrt{n}\,B_2^n.$ 
Then there exists an even continuous probability density $h$ on $\R^n$
such that
\begin{align}
    \supp h
    &\subset
    U\setminus
    \operatorname{int}\left(\left(1+\frac{1}{n}\right)Q\right),
    \label{eq:abstract-h-support}\\
    \int_{\xi^\perp}h(x)\,d\lambda_{\xi^\perp}(x)
    &\geq \frac{c}{v}
    \qquad\text{for every }\xi\in\Sph^{n-1}.
    \label{eq:abstract-h-sections}
\end{align}
\end{proposition}

\begin{proof}
Set $R=Av\sqrt n$, so that $U=RB_2^n$. Define continuous functions $\alpha,\beta:[0,\infty)\to[0,1]$ by
\[
 \alpha(s)=
 \begin{cases}
  1,&0\leq s\leq1-\frac1n,\\
  n(1-s),&1-\frac1n\leq s\leq1,\\
  0,&s\geq1,
 \end{cases}
 \qquad
 \beta(s)=
 \begin{cases}
  0,&0\leq s\leq1+\frac1n,\\
  n(s-1)-1,&1+\frac1n\leq s\leq1+\frac2n,\\
  1,&s\geq1+\frac2n.
 \end{cases}
\]
In what follows, the cutoff $\alpha$ confines the density to $U$, whereas $\beta$ removes an
open neighborhood of $Q$.  The transition widths are of order $1/n$ because
dilating an $n$-dimensional body by $1+O(1/n)$ changes its volume by only an
absolute factor.

Write $
       \norm{x}_Q=\inf\{a>0:x\in aQ\}$
for the Minkowski functional of $Q$.  Since $Q$ is an origin-symmetric convex body, this
functional is a norm; in particular, it is continuous.  Define
\begin{equation}\label{eq:abstract-h0}
       h_0(x)
       =
       \alpha\!\left(\frac{\abs x}{R}\right)
       \beta\!\left(\norm{x}_Q\right).
\end{equation}
The function $h_0$ is even, continuous, and nonnegative.  Since $\alpha=0$
on $[1,\infty)$, it is supported in $U$.  Since $\beta=0$ on $[0,1+1/n]$, we have $
       h_0=0$ on $\left(1+\frac1n\right)Q$. 
In particular, $h_0$ vanishes on an open neighborhood of $Q$, since
\[
       Q\subset
       \operatorname{int}\left(\left(1+\frac1n\right)Q\right).
\]
Moreover,
\[
       \supp h_0
       \subset
       U\setminus
       \operatorname{int}\left(\left(1+\frac1n\right)Q\right)
       \subset U\setminus Q.
\]
The geometry of the cutoff construction is illustrated schematically in  Figure~\ref{fig:annular-density}.

\begin{figure}[ht]
\centering
\begin{tikzpicture}[scale=0.86,>=Latex]
\def\Qshape{plot[smooth cycle,tension=0.78] coordinates {
(-4.0,0) (-2.8,1.05) (-1.35,1.48) (0,1.58)
(1.35,1.48) (2.8,1.05) (4.0,0)
(2.8,-1.05) (1.35,-1.48) (0,-1.58)
(-1.35,-1.48) (-2.8,-1.05)}}
\def\Qplus{plot[smooth cycle,tension=0.78] coordinates {
(-4.28,0) (-3.00,1.12) (-1.44,1.58) (0,1.69)
(1.44,1.58) (3.00,1.12) (4.28,0)
(3.00,-1.12) (1.44,-1.58) (0,-1.69)
(-1.44,-1.58) (-3.00,-1.12)}}
\def\Qplusplus{plot[smooth cycle,tension=0.78] coordinates {
(-4.52,0) (-3.16,1.19) (-1.52,1.67) (0,1.78)
(1.52,1.67) (3.16,1.19) (4.52,0)
(3.16,-1.19) (1.52,-1.67) (0,-1.78)
(-1.52,-1.67) (-3.16,-1.19)}}

\fill[gray!15] (0,0) circle (2.55);
\fill[white] \Qplus;

\begin{scope}
\clip (0,0) circle (2.30);
\fill[gray!45,even odd rule]
(-5,-4) rectangle (5,4) \Qplusplus;
\end{scope}

\draw[very thick] (0,0) circle (2.55);        
\draw[thick] \Qshape;                         
\draw[dashed] \Qplus;                         
\draw[densely dotted] (0,0) circle (2.30);    
\draw[densely dotted] \Qplusplus;             

\node at (0,0) {$Q$};
\node[above right] at (1.68,1.88) {$U$};

\node[fill=white,inner sep=2pt] (support) at (0,2.92)
{possible support of $h_0$};
\draw[->] (support.south) -- (0.35,2.33);

\node[fill=white,inner sep=2pt] (one) at (-3.35,-2.15)
{region where $h_0=1$};
\draw[->] (one.north east) -- (-1.65,-1.88);
\end{tikzpicture}
\caption{A schematic two-dimensional section of the annular construction.
The light-gray region is
$U\setminus\operatorname{int}((1+1/n)Q)$, which contains the support of
$h_0$.  The dark-gray region is
$(1-1/n)U\setminus(1+2/n)Q$, where $h_0=1$.
We note that the body $Q$ and the Euclidean ball $U$ need not be nested.}
\label{fig:annular-density}
\end{figure}
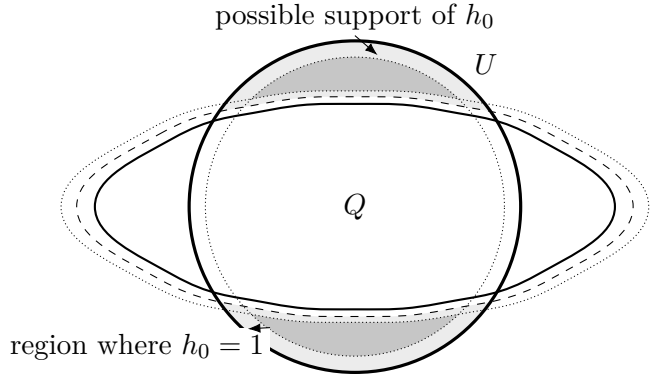

\noindent Next, $h_0=1$ on
$
   \left(1-\frac1n\right)RB_2^n
   \setminus
   \left(1+\frac2n\right)Q.
$
Consequently, 
\begin{equation}\label{eq:abstract-h0-section}
   \int_{\xi^\perp}h_0
   \geq{}
   \abs{B_2^{n-1}}
   \left[\left(1-\frac1n\right)R\right]^{n-1}-
   \left(1+\frac2n\right)^{n-1}\abs{Q\cap \xi^\perp}.
\end{equation}
We compare the two terms in the right-hand side.  The inradius assumption in
\eqref{eq:abstract-Q-geometry} and
Lemma~\ref{lem:central-section-volume} give
\[
   \abs{Q\cap \xi^\perp}
   \leq
   \frac{n}{2r_0\sqrt n}\abs Q
   \leq
   \frac{\sqrt n}{2r_0}(C_0v)^n.
\]
Since $(1+2/n)^{n-1}\leq e^2$, the second term on the right-hand side of
\eqref{eq:abstract-h0-section} is at most
\begin{equation}\label{eq:abstract-bad-term}
       C_1\sqrt n\,(C_0v)^n,
\end{equation}
where $C_1$ depends only on $r_0$.

On the other hand, Lemma \ref{lem:ball-estimates} and the identity $R=Av\sqrt n$ imply
\begin{equation}\label{eq:abstract-good-term}
 \abs{B_2^{n-1}}
 \left[\left(1-\frac1n\right)R\right]^{n-1}
 \geq(c_1Av)^{n-1}
\end{equation}
for an absolute constant $c_1>0$.  The ratio of
\eqref{eq:abstract-bad-term} to \eqref{eq:abstract-good-term} is at most
$
       C_2\sqrt n\,v
       \left(\frac{C_0}{c_1A}\right)^{n-1}.
$
Because $v\leq\sqrt n$, this is bounded by
$$
       C_2n
       \left(\frac{C_0}{c_1A}\right)^{n-1}.
$$
Choose $A_*$ so large that $C_0/(c_1A_*)<1/2$.  For every fixed
$A\geq A_*$, the last expression tends to zero.  Hence, for all sufficiently
large $n$, the second term in
\eqref{eq:abstract-h0-section} is at most one half of the first, and
\begin{equation}\label{eq:abstract-h0-lower}
   \int_{\xi^\perp}h_0
   \geq
   \frac12\abs{B_2^{n-1}}
   \left[\left(1-\frac1n\right)R\right]^{n-1}.
\end{equation}
The lower bound \eqref{eq:abstract-h0-lower} shows that $h_0$ is not
identically zero; hence $\int_{\R^n}h_0(x)\,dx>0$.
Moreover, $0\leq h_0\leq1$ and $\supp h_0\subset U$, hence
\[
       \int_{\R^n}h_0(x)\,dx\leq\abs U=\abs{B_2^n}R^n.
\]
Define $h=h_0 /\int h_0$.  Then $h$ is an even continuous probability density and
\eqref{eq:abstract-h-support} holds. Finally,
\eqref{eq:abstract-h0-lower}, Lemma \ref{lem:ball-estimates}  and $(1-1/n)^{n-1}\geq1/4$ give
\[
   \int_{\xi^\perp}h
   \geq
   \frac{\abs{B_2^{n-1}}}
        {8\abs{B_2^n}R}\geq
   c_2\frac{\sqrt n}{Av\sqrt n}
   =\frac{c_2/A}{v}.
\]
This proves \eqref{eq:abstract-h-sections}.
\end{proof}

\begin{proposition}[The gluing principle]\label{prop:gluing-principle}
Let $1\leq v\leq\sqrt n$, and let $Q,U\subset\R^n$ be origin-symmetric convex
bodies.  Suppose that there are even continuous probability densities $p,h$
on $\R^n$ and positive absolute constants $a,b$ such that
\begin{align}
       \supp p&\subset Q\setminus U,\qquad \int_{\xi^\perp}p
       \leq\frac{a}{\sqrt n} \qquad\text{for all }\xi\in\Sph^{n-1},
          \\
       \supp h&\subset U\setminus Q,\qquad
       \int_{\xi^\perp}h
       \geq\frac{b}{v}
       \qquad\text{for all }\xi\in\Sph^{n-1}.
\end{align}
Then there is an even, continuous, strictly positive, integrable density $f$
such that
\[
       \mu_f(Q\cap\xi^\perp)
       <
       \mu_f(U\cap\xi^\perp)
       \qquad\text{for all }\xi\in\Sph^{n-1}
\]
and
\begin{equation}\label{eq:abstract-mass-ratio}
       \frac{\mu_f(Q)}{\mu_f(U)}
       \geq c\frac{\sqrt n}{v},
\end{equation}
where $c>0$ depends only on $a$ and $b$.
\end{proposition}

\begin{proof}
Set
\begin{equation}\label{eq:abstract-lambda}
       \lambda=\frac{2av}{b\sqrt n}
\end{equation}
and consider first $
       f_0=p+\lambda h.
$

The support separation used below is illustrated in
Figure~\ref{fig:support-separation}. The support assumptions imply
\[
       \mu_{f_0}(Q)=1,
       \qquad
       \mu_{f_0}(U)=\lambda.
\]
Indeed, $p$ is supported in $Q$ and vanishes on $U$, while $h$ is supported
in $U$ and vanishes on $Q$.  Similarly, for every $\xi\in\Sph^{n-1}$,
\[
       \mu_{f_0}(Q\cap\xi^\perp)
       =\int_{\xi^\perp}p
       \leq\frac{a}{\sqrt n},
\]
whereas
\[
       \mu_{f_0}(U\cap\xi^\perp)
       =\lambda\int_{\xi^\perp}h
       \geq\lambda\frac{b}{v}
       =\frac{2a}{\sqrt n}.
\]
Thus the required section comparison already holds with a factor-two margin. The function $f_0$ need not be strictly positive everywhere.  Choose
$\varepsilon>0$ so small that
\begin{equation}\label{eq:abstract-epsilon}
       \varepsilon\leq\frac{\lambda}{2}
       \qquad \mbox{  and  } \qquad
       \frac{\varepsilon}{\sqrt{2\pi}}
       \leq\frac{a}{2\sqrt n}.
\end{equation}
Define $f=f_0+\varepsilon\gauss.$
Then $f$ is even, continuous, integrable, and strictly positive on $\R^n$.
By \eqref{eq:Gaussian-affine-section},
\[
       \int_{\xi^\perp}\gauss=\frac1{\sqrt{2\pi}}.
\]
Since both Gaussian integrals below are nonnegative,
\[
       \int_{U\cap\xi^\perp}\gauss-
       \int_{Q\cap\xi^\perp}\gauss
       \geq-\frac1{\sqrt{2\pi}}.
\]
Therefore
$$
   \mu_f(U\cap\xi^\perp)-\mu_f(Q\cap\xi^\perp)
   \geq
   \frac{2a}{\sqrt n}-\frac{a}{\sqrt n}
   -\frac{\varepsilon}{\sqrt{2\pi}}
   \geq\frac{a}{2\sqrt n}>0.
$$
Finally, $\mu_f(Q)\geq1$, while
\[
       \mu_f(U)
       =\lambda+\varepsilon\int_U\gauss
       \leq\lambda+\varepsilon
       \leq\frac32\lambda.
\]
It follows from \eqref{eq:abstract-lambda} that
\[
       \frac{\mu_f(Q)}{\mu_f(U)}
       \geq\frac{2}{3\lambda}
       =\frac{b}{3a}\frac{\sqrt n}{v},
\]
which proves \eqref{eq:abstract-mass-ratio}.
\end{proof}

\section{{First application: the one-scale lower bound}}
\label{sec:one-application}

\begin{proof}[Proof of Theorem~\ref{thm:one-scale}]
Let $T$ be the {one-scale} body from
\eqref{eq:T-definition}.  By
\eqref{eq:T-inradius} and \eqref{eq:T-volume},
\[
       3\sqrt n\,B_2^n\subset T,
       \qquad
       \abs T^{1/n}\leq C_0\sqrt{\log n}.
\]
Set $  v=\sqrt{\log n}.$ 
For $n\geq3$, one has $1\leq v\leq\sqrt n$.  Choose a fixed absolute constant
$A\geq A_*$ as in Proposition~\ref{prop:annular-density}, and set
\[
       L=Av\sqrt n\,B_2^n
        =A\sqrt{n\log n}\,B_2^n,
\]
and apply Proposition~\ref{prop:one-outer-density} with this value of $A$.
It provides an even continuous probability density $p$ satisfying
\[
       \supp p\subset T\setminus L,
       \qquad
       \int_{\xi^\perp}p\leq\frac{C}{\sqrt n}
       \quad\text{for all }\xi\in\Sph^{n-1}.
\]
Proposition~\ref{prop:annular-density}, applied to $Q=T$ and $U=L$, provides
an even continuous probability density $h$, supported in $L\setminus T$, and
satisfying
\[
       \int_{\xi^\perp}h\geq\frac{c}{\sqrt{\log n}}
       \quad\text{for all }\xi\in\Sph^{n-1}.
\]
All the hypotheses of Proposition~\ref{prop:gluing-principle} are now
satisfied.  It follows that there is an even, continuous, strictly positive,
integrable density $f$ for which
$
       \mu_f(T\cap\xi^\perp)
       <
       \mu_f(L\cap\xi^\perp)$ for all  $\xi\in\Sph^{n-1},
$
while
\[
       \frac{\mu_f(T)}{\mu_f(L)}
       \geq c\sqrt{\frac{n}{\log n}}.
\]
This proves the theorem for all sufficiently large $n$.  For the remaining
finitely many dimensions, decrease the absolute constant $c$ and use
$\calC_n\geq1$, which follows by taking $K=L$ in
Definition~\ref{def:Cn}.
\end{proof}

\section{The Klartag--Livshyts input and support separation}
\label{sec:KL}

To remove the factor $\sqrt{\log n}$ lost in the one-scale
construction, we use the sharp random-rounding construction of Klartag and
Livshyts \cite{KL20}, only through the following consequence of their proof.

\begin{proposition}[Klartag--Livshyts construction]
\label{prop:KL-black-box}
There are absolute constants \(r_0,C_0,c_0,a_0>0\) such that, for every
sufficiently large \(n\), there exist an origin-symmetric convex body
\(T\subset\R^n\) and an even \(C^\infty\) probability density \(g\) on
\(\R^n\) satisfying
\begin{equation}\label{eq:KL-black-box-geometry}
       r_0\sqrt n\,B_2^n\subset T,
       \qquad
       \abs T^{1/n}\leq C_0,
\end{equation}
\begin{equation}\label{eq:KL-black-box-mass}
       \int_Tg(x)\,dx\geq c_0,
\end{equation}
and
\begin{equation}\label{eq:KL-black-box-sections}
       \int_{\xi^\perp}g(x)\,d\lambda_{\xi^\perp}(x)
       \leq\frac{a_0}{\sqrt n}
       \qquad
       \text{for every }\xi\in\Sph^{n-1}.
\end{equation}
\end{proposition}

\begin{proof}
These properties are contained in
\cite[proof of Theorem~1.1, Steps~1--3]{KL20}.  In the notation used there,
the body is \(T=C_6K\).  The definition of \(K\), which includes the points
\(\{\pm ne_i:1\leq i\leq n\}\), gives
\(\sqrt nB_2^n\subset K\).  Step~1 gives
\(\abs{C_6K}\leq C_0^n\), Step~2 proves
\eqref{eq:KL-black-box-sections} (in fact, uniformly over all affine
hyperplanes), and Step~3 gives
\(\int_{C_6K}g(x)\,dx\geq1/2\).  Their density is the convolution of the
standard Gaussian density with an even discrete probability measure; hence
it is even, smooth, strictly positive, and has integral one.  Renaming the
absolute constants gives the stated formulation.
\end{proof}

The next lemma converts the small central-hyperplane integrals into the
support separation required by the gluing argument.  It is independent of
the particular construction of \(g\).

\begin{lemma}[Small sections force mass away from the origin]
\label{lem:spherical-support-separation}
Let \(n\geq2\), let $q:\R^n\to[0,\infty)$ be a Borel
integrable function, and suppose that
\[
       \int_{\xi^\perp}q(x)\,d\lambda_{\xi^\perp}(x)
       \leq\frac{a}{\sqrt n}
       \qquad
       \text{for every }\xi\in\Sph^{n-1}.
\]
Then, for every \(R>0\),
\begin{equation}\label{eq:mass-in-ball}
       \int_{RB_2^n}q(x)\,dx
       \leq C\,\frac{aR}{n},
\end{equation}
where \(C>0\) is an absolute constant.
\end{lemma}

\begin{proof}
Let \(\bar\sigma_{n-1}\) be the normalized surface measure on
\(\Sph^{n-1}\).  Spherical averaging of the central Radon transform gives
\begin{equation}\label{eq:averaged-Radon}
 \int_{\Sph^{n-1}}
 \left(\int_{\xi^\perp}q(x)\,d\lambda_{\xi^\perp}(x)\right)
 d\bar\sigma_{n-1}(\xi)
 =
 \frac{\sigma_{n-2}(\Sph^{n-2})}
            {\sigma_{n-1}(\Sph^{n-1})}\int_{\R^n}\frac{q(x)}{\abs x}\,dx.
\end{equation}
For completeness, \eqref{eq:averaged-Radon} follows by writing the
hyperplane integral in polar coordinates inside \(\xi^\perp\), applying
Fubini on the incidence set
$
       \{(\xi,\theta)\in\Sph^{n-1}\times\Sph^{n-1}:
          \ip{\xi}{\theta}=0\},
$
and using rotational invariance.  The constant is obtained by applying the
identity to a radial function. 
Since \(q\geq0\) and \(\abs{x}\leq R\) for \(x\in RB_2^n\),
\[
   \int_{\R^n}\frac{q(x)}{\abs{x}}\,dx
   \geq
   \int_{RB_2^n}\frac{q(x)}{\abs{x}}\,dx
   \geq
   \frac1R\int_{RB_2^n}q(x)\,dx.
\]
Therefore, by  Lemma \ref{lem:ball-estimates},
\begin{equation}\label{eq:ball-mass-lower}
   \frac{\sigma_{n-2}(\Sph^{n-2})}
        {\sigma_{n-1}(\Sph^{n-1})}\int_{\R^n}\frac{q(x)}{\abs{x}}\,dx
   \geq
   \frac{c\sqrt n}{R}\int_{RB_2^n}q(x)\,dx.
\end{equation}
On the other hand, the spherical-averaging identity
\eqref{eq:averaged-Radon} and the assumed section bound imply
\[
    \frac{\sigma_{n-2}(\Sph^{n-2})}
        {\sigma_{n-1}(\Sph^{n-1})}\int_{\R^n}\frac{q(x)}{\abs{x}}\,dx
   =
   \frac1{\sigma_{n-1}(\Sph^{n-1})}
   \int_{\Sph^{n-1}}
      \left(\int_{\xi^\perp}q(x)\,
      d\lambda_{\xi^\perp}(x)\right)
   d\sigma_{n-1}(\xi) \leq \frac{a}{\sqrt n}.
\]
Combining this with \eqref{eq:ball-mass-lower}, we obtain  \eqref{eq:mass-in-ball}.
\end{proof}

We now use continuous cutoffs to turn the Klartag--Livshyts density into a
probability density with genuinely separated support.

\begin{proposition}[The sharp outer density]
\label{prop:KL-outer-density}
Fix \(A\geq1\).  For all sufficiently large \(n\), the body \(T\) in
Proposition~\ref{prop:KL-black-box} may be chosen so that, with $       L=A\sqrt n\,B_2^n,$ 
there exists an even continuous probability density \(p\) on \(\R^n\)
such that  $\supp p\subset T\setminus L$ and 
\begin{equation}
       \sup_{\xi\in\Sph^{n-1}}
\int_{\xi^\perp}p(x)\,d\lambda_{\xi^\perp}(x)
       \leq\frac{C}{\sqrt n}.
       \label{eq:KL-p-sections}
\end{equation}
Here \(C>0\) is an absolute constant, once \(A\) is fixed.
\end{proposition}

\begin{proof}
Choose \(T\) and \(g\) as in Proposition~\ref{prop:KL-black-box}, and set
\[
       q=g\1_T.
\]
By \eqref{eq:KL-black-box-sections}, \(q\) satisfies the hypothesis of
Lemma~\ref{lem:spherical-support-separation} with \(a=a_0\).
Applying \eqref{eq:mass-in-ball} with \(R=3A\sqrt n\), we obtain
\[
       \int_{T\cap3A\sqrt nB_2^n}g(x)\,dx
       \leq\frac{C_A}{\sqrt n}.
\]
Together with \eqref{eq:KL-black-box-mass}, this implies, for all
sufficiently large \(n\), that
\begin{equation}\label{eq:KL-outer-mass}
       \int_{T\setminus3A\sqrt nB_2^n}g(x)\,dx
       \geq\frac{c_0}{2}.
\end{equation}

Let \(\eta:\R^n\to[0,1]\) be the even radial function that equals zero on
\(2A\sqrt nB_2^n\), equals one outside \(3A\sqrt nB_2^n\), and is linear
in \(\abs x\) on the intervening annulus.  For \(0<\delta<1/3\), define
\[
 \omega_\delta(x)=
 \begin{cases}
  1,&\norm{x}_T\leq1-2\delta,\\
  \displaystyle\frac{1-\delta-\norm{x}_T}{\delta},
       &1-2\delta\leq\norm{x}_T\leq1-\delta,\\
  0,&\norm{x}_T\geq1-\delta.
 \end{cases}
\]
Then \(\omega_\delta\) is even and continuous, and its support is contained
in \((1-\delta)T\).  Since the boundary of \(T\) has Lebesgue measure zero,
dominated convergence and \eqref{eq:KL-outer-mass} show that, for a
sufficiently small \(\delta>0\),
\[
       m_\delta
       :=
       \int_{\R^n}g(x)\omega_\delta(x)\eta(x)\,dx
       \geq\frac{c_0}{4}.
\]
Define
\[
       p(x)=\frac{g(x)\omega_\delta(x)\eta(x)}{m_\delta}.
\]
This is an even continuous probability density.  Its two cutoffs give
\[
       \supp p
       \subset
       (1-\delta)T\setminus\operatorname{int}
       \left(2A\sqrt nB_2^n\right)
       \subset T\setminus L.
\]
Moreover,  $0\leq p\leq\frac{4}{c_0}g\1_T.$ 
Integrating this inequality over $\xi^\perp$ and using
\eqref{eq:KL-black-box-sections} proves
\eqref{eq:KL-p-sections}.
\end{proof}

\section{Second application: the sharp lower bound}
\label{sec:sharp-application}

\begin{proof}[Proof of Theorem~\ref{thm:main}]
Let \(r_0,C_0\) be the constants in
Proposition~\ref{prop:KL-black-box}.  Choose a fixed absolute constant
$A\geq\max\{1,A_*(r_0,C_0)\}$ as in
Proposition~\ref{prop:annular-density}, choose
$T$ as in Proposition~\ref{prop:KL-outer-density}, and set
$
       L=A\sqrt n\,B_2^n.
$ 
By \eqref{eq:KL-black-box-geometry}, the assumptions
\eqref{eq:abstract-Q-geometry} hold with $Q=T$,  $U=L$ and  $v=1$. 
Proposition~\ref{prop:annular-density} gives an even continuous probability
density \(h\), supported in \(L\setminus T\), and satisfying
\[
       \int_{\xi^\perp}h\geq b
       \qquad\text{for every }\xi\in\Sph^{n-1},
\]
where \(b>0\) is absolute.  Proposition~\ref{prop:KL-outer-density} gives an
even continuous probability density \(p\) such that
\[
       \supp p\subset T\setminus L,
       \qquad
       \int_{\xi^\perp}p\leq\frac{a}{\sqrt n}
       \qquad\text{for every }\xi\in\Sph^{n-1},
\]
where \(a>0\) is absolute.

Apply Proposition~\ref{prop:gluing-principle} with \(v=1\).  We obtain an
even, continuous, strictly positive, integrable density $f$ such that
$
       \mu_f(T\cap\xi^\perp)
       <
       \mu_f(L\cap\xi^\perp)$ for every $\xi\in\Sph^{n-1}$ 
and
\[
       \frac{\mu_f(T)}{\mu_f(L)}
       \geq c\sqrt n.
\]
This proves the lower estimate in Theorem~\ref{thm:main} for all sufficiently
large \(n\).  For the remaining finitely many dimensions, decrease \(c\) and
use \(\calC_n\geq1\).  The upper estimate is
\eqref{eq:known-upper}.
\end{proof}
\begin{remark}\label{rem:properties}
The density in the sharp example cannot be log-concave.  Indeed, every
log-concave density is, in particular, $(-1/n)$-concave.  Therefore
 \cite[Theorem~4]{KZ15} (see also \cite{WU20} and a more general inequality in \cite{Tz}) bounds the
isomorphic Busemann--Petty constant for
even log-concave densities, up to an absolute factor, by the maximal
isotropic constant in dimension $n$.  The dimension-free estimate for
isotropic constants proved in \cite[Theorem~1.2]{KL25} makes this restricted
isomorphic constant absolute.  This is incompatible, for large $n$, with
the ratio of order $\sqrt n$ obtained here. We also note that, the two principal components of $f$ have deliberately separated supports.
The small Gaussian term in the proof of
Proposition~\ref{prop:gluing-principle} is used only to enforce strict
positivity.
\end{remark}

\section*{Acknowledgments}

The authors used OpenAI's ChatGPT 5.6 Sol to assist in checking calculations,
improving the exposition, and preparing the manuscript.

\address{Alexander Koldobsky, Department of Mathematics,
University of Missouri, Columbia, MO 65211, USA}
\email{koldobskiya@missouri.edu}

\address{Artem Zvavitch, Department of Mathematical Sciences,
Kent State University, Kent, OH 44242, USA}
\email{zvavitch@math.kent.edu}


\small
\begin{thebibliography}{99}

\bibitem{AGM21}
S.~Artstein-Avidan, A.~Giannopoulos, and V.~D. Milman,
\emph{Asymptotic Geometric Analysis. Part II},
Mathematical Surveys and Monographs, vol.~261,
American Mathematical Society, Providence, RI, 2021.
\href{https://doi.org/10.1090/surv/261}
{doi:10.1090/surv/261}.


\bibitem{Be26}
P.~Bizeul,
\emph{The slicing conjecture via small ball estimates},
\href{https://arxiv.org/abs/2501.06854}{arXiv:2501.06854}, 2025.

\bibitem{BKK18}
S.~G. Bobkov, B.~Klartag, and A.~Koldobsky,
\emph{Estimates for moments of general measures on convex bodies},
Proc. Amer. Math. Soc. \textbf{146} (2018), no.~11, 4879--4888.
\href{https://doi.org/10.1090/proc/14119}{doi:10.1090/proc/14119};
\href{https://arxiv.org/abs/1712.05949}{arXiv:1712.05949}.

\bibitem{Bourgain86}
J.~Bourgain,
\emph{On high-dimensional maximal functions associated to convex bodies},
Amer. J. Math. \textbf{108} (1986), no.~6, 1467--1476.
\href{https://doi.org/10.2307/2374532}{doi:10.2307/2374532}.

\bibitem{BP56}
H.~Busemann and C.~M. Petty,
\emph{Problems on convex bodies},
Math. Scand. \textbf{4} (1956), 88--94.
\href{https://doi.org/10.7146/math.scand.a-10457}
{doi:10.7146/math.scand.a-10457}.

\bibitem{CGL}
G.~Chasapis, A.~Giannopoulos and D.~Liakopoulos,
\emph{Estimates for measures of lower dimensional sections of convex bodies},
Adv. Math. \textbf{306} (2017), 880--904.
\href{https://doi.org/10.1016/j.aim.2016.10.035}
{doi:10.1016/j.aim.2016.10.035}

\bibitem{Chen21}
Y.~Chen,
\emph{An almost constant lower bound of the isoperimetric coefficient in the
KLS conjecture},
Geom. Funct. Anal. \textbf{31} (2021), 34--61.
\href{https://doi.org/10.1007/s00039-021-00556-4}
{doi:10.1007/s00039-021-00556-4};
\href{https://arxiv.org/abs/2011.13661}{arXiv:2011.13661}.

\bibitem{GHS26}
D.~Galicer, J.~Haddad, and J.~Singer,
\emph{Extensions of the Busemann--Petty problem for arbitrary measures},
\href{https://arxiv.org/abs/2512.13555}{arXiv:2512.13555v2}, 2026.

\bibitem{Gardner06}
R.~J. Gardner,
\emph{Geometric Tomography}, second ed.,
Encyclopedia of Mathematics and its Applications, vol.~58,
Cambridge University Press, Cambridge, 2006.
\href{https://doi.org/10.1017/CBO9781107341029}
{doi:10.1017/CBO9781107341029}.

\bibitem{GKS99}
R.~J. Gardner, A.~Koldobsky, and T.~Schlumprecht,
\emph{An analytic solution to the Busemann--Petty problem on sections of
convex bodies},
Ann. of Math. (2) \textbf{149} (1999), no.~2, 691--703.
\href{https://doi.org/10.2307/120978}{doi:10.2307/120978};
\href{https://arxiv.org/abs/math/9903200}{arXiv:math/9903200}.

\bibitem{Glu89}
E.~D. Gluskin,
\emph{Extremal properties of orthogonal parallelepipeds and their applications
to the geometry of Banach spaces},
Math. USSR-Sb. \textbf{64} (1989), no.~1, 85--96.
\href{https://doi.org/10.1070/SM1989v064n01ABEH003295}
{doi:10.1070/SM1989v064n01ABEH003295}.

\bibitem{Guan24}
Q.~Guan,
\emph{A note on Bourgain's slicing problem},
\href{https://arxiv.org/abs/2412.09075}{arXiv:2412.09075}, 2024.

\bibitem{KK18}
B.~Klartag and A.~Koldobsky,
\emph{An example related to the slicing inequality for general measures},
J. Funct. Anal. \textbf{274} (2018), no.~7, 2089--2112.
\href{https://doi.org/10.1016/j.jfa.2017.08.025}
{doi:10.1016/j.jfa.2017.08.025};
\href{https://arxiv.org/abs/1706.01132}{arXiv:1706.01132}.

\bibitem{KL20}
B.~Klartag and G.~V. Livshyts,
\emph{The lower bound for Koldobsky's slicing inequality via random rounding},
in \emph{Geometric Aspects of Functional Analysis. Israel Seminar
(GAFA) 2017--2019, Vol.~II}, Lecture Notes in Math. \textbf{2266},
Springer, Cham, 2020, pp.~43--63.
\href{https://doi.org/10.1007/978-3-030-46762-3_2}
{doi:10.1007/978-3-030-46762-3\_2}.
Corrected version:
\href{https://arxiv.org/abs/1810.06189}
{arXiv:1810.06189v4} (July 18, 2023).

\bibitem{KL22}
B.~Klartag and J.~Lehec,
\emph{Bourgain's slicing problem and KLS isoperimetry up to polylog},
Geom. Funct. Anal. \textbf{32} (2022), no.~5, 1134--1159.
\href{https://doi.org/10.1007/s00039-022-00611-4}
{doi:10.1007/s00039-022-00611-4};
\href{https://arxiv.org/abs/2203.15551}{arXiv:2203.15551}.

\bibitem{KL25}
B.~Klartag and J.~Lehec,
\emph{Affirmative resolution of Bourgain's slicing problem using Guan's
bound},
Geom. Funct. Anal. \textbf{35} (2025), no.~4, 1147--1168.
\href{https://doi.org/10.1007/s00039-025-00718-w}
{doi:10.1007/s00039-025-00718-w};
\href{https://arxiv.org/abs/2412.15044}{arXiv:2412.15044}.

\bibitem{Koldobsky05}
A.~Koldobsky,
\emph{Fourier Analysis in Convex Geometry},
Mathematical Surveys and Monographs, vol.~116,
American Mathematical Society, Providence, RI, 2005.
\href{https://doi.org/10.1090/surv/116}{doi:10.1090/surv/116}.

\bibitem{Koldobsky14}
A.~Koldobsky,
\emph{A $\sqrt n$ estimate for measures of hyperplane sections of convex
bodies},
Adv. Math. \textbf{254} (2014), 33--40.
\href{https://doi.org/10.1016/j.aim.2013.12.029}
{doi:10.1016/j.aim.2013.12.029};
\href{https://arxiv.org/abs/1309.5271}{arXiv:1309.5271}.

\bibitem{Koldobsky15}
A.~Koldobsky,
\emph{Slicing inequalities for measures of convex bodies},
Adv. Math. \textbf{283} (2015), 473--488.
\href{https://doi.org/10.1016/j.aim.2015.07.019}
{doi:10.1016/j.aim.2015.07.019};
\href{https://arxiv.org/abs/1412.8550}{arXiv:1412.8550}.

\bibitem{KPZ22}
A.~Koldobsky, G.~Paouris, and A.~Zvavitch,
\emph{Measure comparison and distance inequalities for convex bodies},
Indiana Univ. Math. J. \textbf{71} (2022), no.~1, 391--407.
\href{https://doi.org/10.1512/iumj.2022.71.8838}
{doi:10.1512/iumj.2022.71.8838};
\href{https://arxiv.org/abs/1912.00880}{arXiv:1912.00880}.

\bibitem{KZ15}
A.~Koldobsky and A.~Zvavitch,
\emph{An isomorphic version of the Busemann--Petty problem for arbitrary
measures},
Geom. Dedicata \textbf{174} (2015), 261--277.
\href{https://doi.org/10.1007/s10711-014-0016-x}
{doi:10.1007/s10711-014-0016-x};
\href{https://arxiv.org/abs/1405.0567}{arXiv:1405.0567}.

\bibitem{MP89}
V.~D. Milman and A.~Pajor,
\emph{Isotropic position and inertia ellipsoids and zonoids of the unit ball
of a normed $n$-dimensional space},
in \emph{Geometric Aspects of Functional Analysis (1987--88)},
Lecture Notes in Math. \textbf{1376},
Springer, Berlin, 1989, pp.~64--104.
\href{https://doi.org/10.1007/BFb0090049}{doi:10.1007/BFb0090049}.

\bibitem{NIST}
F.~W.~J. Olver, D.~W. Lozier, R.~F. Boisvert, and C.~W. Clark (eds.),
\emph{NIST Handbook of Mathematical Functions},
Cambridge University Press, New York, 2010.
\href{https://dlmf.nist.gov/5.6}
{DLMF Section 5.6: Gamma-function inequalities}.

\bibitem{Pisier89}
G.~Pisier,
\emph{The Volume of Convex Bodies and Banach Space Geometry},
Cambridge Tracts in Mathematics, vol.~94,
Cambridge University Press, Cambridge, 1989.
\href{https://doi.org/10.1017/CBO9780511662454}
{doi:10.1017/CBO9780511662454}.

\bibitem{Tz}
N.~Tziotziou,
\emph{Inequalities for sections and projections of log-concave functions},
J. Geom. Analysis \textbf{36} (2026), Paper No. 145.
\href{https://doi.org/10.1007/s12220-026-02388-y}
{doi.org/10.1007/s12220-026-02388-y}

\bibitem{WU20}
D.~Wu,
\emph{The isomorphic Busemann--Petty problem for $s$-concave measures},
Geom. Dedicata \textbf{204} (2020), 131--148.
\href{https://doi.org/10.1007/s10711-019-00446-0}
{doi:10.1007/s10711-019-00446-0}.


\bibitem{Zhang99}
G.~Zhang,
\emph{A positive solution to the Busemann--Petty problem in $\R^4$},
Ann. of Math. (2) \textbf{149} (1999), no.~2, 535--543.
\href{https://doi.org/10.2307/120974}{doi:10.2307/120974};
\href{https://arxiv.org/abs/math/9903205}{arXiv:math/9903205}.

\bibitem{Zvavitch05}
A.~Zvavitch,
\emph{The Busemann--Petty problem for arbitrary measures},
Math. Ann. \textbf{331} (2005), no.~4, 867--887.
\href{https://doi.org/10.1007/s00208-004-0611-5}
{doi:10.1007/s00208-004-0611-5};
\href{https://arxiv.org/abs/math/0406406}{arXiv:math/0406406}.

\end{thebibliography}
\end{document}